\documentclass{amsart}
\usepackage{graphicx}
\usepackage{latexsym}
\usepackage{amsfonts}
\usepackage[all]{xy}
\usepackage{amssymb, mathrsfs, amsfonts, amsmath}
\usepackage{amsbsy}
\usepackage{amsfonts}
\newtheorem{corollary}{Corollary}

\newtheorem{definition}{Definition}
\newtheorem{lemma}{Lemma}
\newtheorem{proposition}{Proposition}
\newtheorem{remark}{Remark}
\newtheorem{theorem}{Theorem}

\numberwithin{equation}{section}

\begin{document}
	
	\title[Charged Hawking mass in the electrostatic space]{Sharp lower bound for the charged Hawking mass in the electrostatic space}
	
	

	\author{Benedito Leandro}

	\address{University of Bras\'ilia, department of mathematics, Bras\'ilia-DF, 70910-900, Brazil} 
	\curraddr{}
	\email{bleandrone@mat.unb.br} 
	\thanks{Benedito Leandro was partially supported by CNPq/Brazil Grant 303157/2022-4 and 403349/2021-4.}

	\author{Guilherme Sabo}
	\address{University of Bras\'ilia, department of mathematics, Bras\'ilia-DF, 70910-900, Brazil}
	\curraddr{}
	\email{guilhermepaes92@hotmail.com}
	\thanks{Guilherme Sabo was partially supported by CAPES/Brazil Grant 88887.702728/2022-00}

	\subjclass[2020]{83C22, 83C05.}
	
	\date{}
	
	\dedicatory{}
	
	\begin{abstract}
		We prove sharp lower bounds for the charged Hawking mass of stable surfaces
in electrostatic space-times in various contexts. An upper bound for the genus of stable surfaces in the electrostatic system is provided. We also study the positivity for the charged Hawking mass of a minimal surface with index one in the electrostatic space-times. A criterion for a CMC surface in the Reissner-Nordstrom deSitter space to be stable is presented.
	\end{abstract}
	
	\maketitle
	
	\section{Introduction}

	A static space-time is a product manifold $\mathrm{M}^4=\mathbb{R}\times M^3$ with metric $\mathrm{g}=-f^2dt^2+g$, for a smooth, positive function $f:M\to(0,\,+\infty)$ and Riemannian metric $g$. Here, $(M^3,\,g)$ is an oriented three-dimensional Riemannian manifold. 
	
	The Einstein-Maxwell equations with cosmological constant $\Lambda$ for $(\mathrm{M}^4,\,\mathrm{g})$ are expressed by the following system
	\begin{equation}
		\left\{\begin{array}{rcll}
			&&\mathrm{Ric}_{\mathrm{g}}-\frac{\mathrm{R}_{\mathrm{g}}}{2}\mathrm{g} + \Lambda\mathrm{g} = 2\left(F\circ F -\frac{1}{4}|F|^2_{\mathrm{g}}\mathrm{g}\right);\nonumber\\\\
			&&dF=0\quad\mbox{and}\quad \textnormal{div}_{\mathrm{g}}F=0,\nonumber
		\end{array}\right.
	\end{equation}
	where $F$ stands for the (Faraday) electromagnetic $(0,\,2)$-tensor and $(F\circ F)_{\alpha\beta} = \mathrm{g}^{\sigma\gamma}F_{\alpha\sigma}F_{\beta\gamma},$ the Greek letters range from $1$ to $4$. 
	Taking $F=fE^{\flat}\wedge dt,$ where $E^{\flat}$ is the dual $1$-form associated with the electric field $E$, and a static space-time we obtain the following system of equations (cf. \cite{chrusciel2017non,tiarlos}).

	\begin{definition}\label{def1}
		Let a Riemannian $3$-manifold $(M^3,\,g)$, $E \in \mathfrak{X}(M)$
		a tangent vector field, and $f\in C^{\infty}(M)$ such that $f > 0$. The Einstein-Maxwell equations with cosmological constant $\Lambda$ for the electrostatic space-time associated to $(M^3,\, g,\, f,\,E)$ are given by
		\begin{equation}\label{s1}
			\begin{array}{rcll}
				\nabla^2f&=&f(\textnormal{Ric}-\Lambda g+2E^\flat\otimes E^\flat-|E|^2g),\\\\
				\Delta f&=&(|E|^2-\Lambda)f,\quad0=\textnormal{div}(E)\quad\mbox{and}\quad 0\,=\,\textnormal{curl}(fE).
			\end{array}
		\end{equation} 
		We say $(M^3,\, g,\, f,\,E)$ is an electrostatic system if $(M^3,\,g)$ is an oriented Riemannian manifold, $f\in C^{\infty}(M)$ and it is not identically zero, $E \in \mathfrak{X}(M)$, and the equations \eqref{s1} are satisfied for some constant
		$\Lambda\in\mathbb{R}$. The system is complete if $(M^3,\,g)$ is complete.
	\end{definition}

	Here, $\textnormal{Ric}$, $\nabla^2$, $\textnormal{div}$ and $\Delta$ stand for the Ricci tensor, Hessian tensor, divergence, and Laplacian concerning the metric $g$, respectively.

	It is well-known that $\textnormal{curl}(fE)=0$ if and only if $d(fE^\flat)=0.$ The smooth function $f$ is called the lapse function (or static potential), the field $E$ is the electric field, and $M^3$ is the three-dimensional manifold. Moreover, $f > 0$ on $M$. If $M$ has a horizon boundary $\partial M$, we assume in addition that $f^{-1}(0)\subseteq\partial M$ (cf. \cite{cederbaum2016uniqueness, chrusciel2017non,tiarlos}). It is interesting to point out that $\textnormal{curl}(fE)=0$ is trivially satisfied when $fE=\nabla\psi$ for some potential smooth function $\psi$ defined on $M$.
	
	Note that taking the contraction of the first equation and combining it with the Laplacian of $f$ in \eqref{s1}, we obtain a useful equation that relates the scalar curvature $R$, the cosmological constant, and the electric field:
	\begin{equation}\label{rrr}
		R=2(|E|^2+\Lambda).
	\end{equation}

	We will explore the Reissner-Nordstrom deSitter (RNdS) space as an example of electrostatic space. The $3+1$-dimensional RNdS space-time is a $3$-parameter family (labeled with a mass $\mathfrak{M}$, a charge $Q$ and cosmological constant $\Lambda$) of static, electrically charged,  solutions to the Einstein equations. Consider $(M^3,\,g_{RNdS},\,f,\,E)$, where $$M^3=I\times\mathbb{S}^2$$
	for some $I\subset\mathbb{R}$ which depends on the roots of the static potential given by
	$$f(r) ^2 = 1-\frac{2\mathfrak{M}}{r}+\frac{Q^2}{r^2}-\frac{\Lambda r^2}{3}.$$ This polynomial equation has four distinct solutions. There is a negative root with no physical relevance. We denote by $r_c> r_{+}>r_{\--}$ the positive roots. 
	The metric and the electric field for RNdS space are given by
	$$g_{RNdS}=f(r)^{-2}dr^2+r^2g_{\mathbb{S}^2}\quad\mbox{and} \quad E=\frac{Q}{r^2}f(r)\partial r.$$
	In the above, $Q,\,\mathfrak{M}\in\mathbb{R}$, and $g_{\mathbb{S}^2}$ stand for the charge, mass, and the standard metric of the unit sphere $\mathbb{S}^2$, respectively. Here, $r$ represents the radial coordinate.
	The set $\{r = r_{\--} \}$ is the inner
	(Cauchy) black hole horizon, $\{r = r_+ \}$ is the outer (Killing) black hole
	horizon, and $\{r = r_c \}$ is the cosmological horizon. We can have solutions possessing some extreme conditions: the cold black hole $r_{\--} = r_+$, the Nariai black hole $r_+ = r_c$, and the ultracold black hole $r_{\--} = r_+ = r_c$, see more details in \cite{tiarlos}. These solutions for the electrostatic system (Definition \ref{def1}) can be seen as a static black hole model with a given Hawking temperature (the reason behind the term ``cold").

	By change of variable (cf. \cite{baltazar2023}), the RNdS metric can be rewritten as 
	\begin{eqnarray}\label{novacoord}
		\tilde{g} = ds^2 + u(s)^2g_{\mathbb{S}^2};\quad [0,\,a]\times\mathbb{S}^2, 
	\end{eqnarray}
	where $u:(0,\,a)\to(r_+,\,r_c)$ is a function that extends continuously to $[0,\,a]$ with $u(0)=r_+$, $u(a)=r_c$, and $ds=\dfrac{dr}{f(r)}$, where $f>0$ for any $r\in(r_+,\,r_c).$
	
	Remember the total electric charge contained within an orientable closed surface $\Sigma$ with unit normal $\nu$, i.e.,
	\begin{eqnarray}\label{chaaaarge}
		Q(\Sigma)=\frac{1}{4\pi}\int_{\Sigma}\langle E,\,\nu\rangle d\sigma, 
	\end{eqnarray}
	and the ADM mass of a Riemannian manifold $(M^3,\,g)$, which is given by
	\begin{eqnarray*}
		\mathfrak{M} = \mathfrak{M}(M,\,g) =\frac{1}{8\pi}\displaystyle\lim_{r\to+\infty}\int_{\mathbb{S}(r)}\displaystyle\sum_{i,\,j=1}^3(\partial_ig_{ij}-\partial_{j}g_{ii})\nu^j,
	\end{eqnarray*}
	where $\mathbb{S}(r)$ is the standard sphere. Here, $\nu^j$ represents one component of the normal vector field of $\mathbb{S}(r).$

	A relevant concept in the electrostatic theory is the photon sphere. A time-like embedded orientable surface $\Sigma$ in $(\mathbb{R}\times M^3,\,-f^2dt^2 + g)$ is called a photon sphere if it is totally umbilical and $f$ is constant on every connected component
	of $\Sigma$. In the three-dimensional Reissner–Nordstrom space (i.e., $\Lambda=0$) with mass parameter $\mathfrak{M} >0$, there is a photon sphere located at the radius
 \begin{eqnarray}\label{raiophoton}
     \frac{3\mathfrak{M}}{2}\left(1 + \sqrt{1-\frac{8Q^2}{9\mathfrak{M}^2}}\right),
 \end{eqnarray}
	whenever $ 9 \mathfrak{M}^2\geq 8 Q^2$ (e.g., subextremal Reissner–Nordstrom space), see \cite{cederbaum2016uniqueness,sophia}. A Reissner–Nordstrom space is called subextremal (extremal, superextremal) if $\mathfrak{M}^2 > Q^2$ (if $\mathfrak{M}^2 = Q^2$, $\mathfrak{M}^2 < Q^2$). The photon sphere appears naturally in our study and plays an important role in our main results.
	
	The Reissner-Nordstrom photon sphere models an embedded submanifold ruled by photons spiraling around the central black hole ``at a fixed distance”. The Reissner–Nordstrom photon spheres are related to the existence of relativistic images in the context of gravitational lensing \cite{cederbaum2016uniqueness,sophia}. Such a photon sphere is a stable constant mean curvature (CMC) surface (see Proposition \ref{CMCstable} below). In fact, any sphere $\mathbb{S}^2(r)$ of radius $r$ in the RNdS is stable if
	\begin{eqnarray*}
		\frac{3\mathfrak{M}}{2}\left(1 - \sqrt{1-\frac{8Q^2}{9\mathfrak{M}^2}}\right)\leq r\leq  \frac{3\mathfrak{M}}{2}\left(1 + \sqrt{1-\frac{8Q^2}{9\mathfrak{M}^2}}\right).
	\end{eqnarray*}

	The concept of stability of a surface is related to its Morse index. The Morse index of a closed surface $\Sigma$ is given by the number of negative eigenvalues of the Jacobi operator $J$, where
	$$J=-\Delta^\Sigma - Ric(\nu,\,\nu) - |A|^2.$$ A constant mean curvature (CMC) surface $\Sigma$ in a Riemannian manifold is stable if the Jacobi operator is non-negative, i.e., the Morse index is zero ($Index(\Sigma)=0$). Here, $\Delta^\Sigma$ and $A$ stand for the Laplacian and the second fundamental form of $\Sigma$. Moreover, $\nu$ represents the normal vector field of $\Sigma.$ In the proof of Theorem \ref{Teo1}, we do not impose any information about the mean curvature $H$ of a closed surface in the electrostatic space $(M^3,\,g,\,f,\,E)$. It can be either non-constant or constant, and it can change signs. In Theorem \ref{Teo1}, we will only impose that $J$ must be non-negative, that is, $Index(\Sigma)=0$. Moreover, the photon sphere will be an important example of this theorem. Stable CMC surfaces are a fundamental tool in studying general relativity. For example, the theory of stable minimal surfaces was essential in proving the positive mass theorem \cite{schoen79}. In fact, the subject of mass will be central to this paper.


	 In Newtonian gravity, we can define the mass of a region by integrating the ``mass density function''. However, defining a similar concept in general relativity is difficult due to the Equivalence Principle (in general relativity, there is no density for gravitation). The concept of mass is a big deal in general relativity. In 1982, Penrose \cite{penrose} listed major open problems in general relativity, and the first problem was to ``find a suitable quasi-local definition of energy-momentum". Mathematically, this problem is a little bit loose. So, several definitions of quasi-local mass have emerged since then in an attempt to establish a well-defined notion of mass in general relativity. For instance, Penrose (1982), Bartnik (1989), Dougan-Manson (1991), Brown-York (1993), Zhang (2006), Liu-Yau (2006), Wang-Yau (2009), Alaee-Khuri-Yau (2023). But what properties can we expect from a quasi-local mass if it is well-defined? To get a good definition of quasi-local mass, certain properties must be satisfied, such as 
	\begin{enumerate}
		\item The mass must be positive for a large class of surfaces.
		\item It should vanish for surfaces in flat space-time.
		\item It should converge to the ADM mass at flat infinity.
		\item It should converge to the Bondi mass at null infinity.
		\item It should satisfy Penrose inequality if the surface enclosed a horizon.
		\item  It should satisfy some stability property for the rigidity case, i.e., it is expected that if we have a sequence of a finite domain and if the defined quasi-local mass of such a sequence of domain converges to zero, this sequence of finite domain converges to a domain in Minkowski space or Euclidean space.
  
  \item Monotonicity property under inclusion, i.e.,  $\forall\,\Omega_1\Subset\Omega_2\subset M$, with $\partial\Omega_1$ outer-minimizing in $\Omega_2$ and $\partial\Omega_2$ outer-minimizing in $M$. Then, $\textnormal{mass}(\Omega_1)\leq \textnormal{mass}(\Omega_2)$.
	\end{enumerate}
	We will define the term ``outer-minimizing" below. See more details about these properties of a quasi-local mass in \cite{alaee2023,christodoulou1986,mondino2022} and in the references therein. As far as we know, there is no universal definition of quasi-local mass in general relativity.

	One of those notions of quasi-local mass, the Hawking mass, has gotten much attention due to its important role in proving the Riemannian Penrose inequality (cf. \cite{huisken2001}). The ``standard'' Hawking mass of a surface $\Sigma^2$ in a given Riemannian manifold $(M^3,\,g)$ is given by
	\begin{eqnarray*}
		\mathfrak{M}_{H}(\Sigma) = \sqrt{\frac{|\Sigma|}{16 \pi}}\left(1 - \frac{1}{16\pi}\int_{\Sigma}H^2d\sigma \right),
	\end{eqnarray*}
 where $|\Sigma|$ and $H$ stand for the area of the surface and its mean curvature with respect to the metric $g$, respectively.
	A straightforward computation shows that if $\Sigma$ is a minimal surface, its Hawking mass is positive.  On the one hand, the Hawking mass of any surface in $\mathbb{R}^3$ is less than or equal to zero, with equality if and only if the surface is a round sphere (cf. \cite{mondino2022} and the references therein). These examples contradict the first item of the nice properties that a quasi-local mass must satisfy. On the other hand, Christodoulou and Yau \cite{christodoulou1986} proved that the Hawking mass is non-negative for stable constant mean curvature spheres in $3$-manifolds with non-negative scalar curvature. This shows that even the concept of Hawking's mass being positive can be a bitter issue to deal with. This paper aims to provide sharp lower bounds for the charged Hawking mass of stable surfaces in electrostatic space-times in various contexts, guaranteeing that the charged Hawking mass satisfies Item (1) under certain conditions.

	In \cite[Proposition 3]{miao2005}, the author proved a lower bound for the Hawking mass of a connected stable CMC surface boundary $\Sigma$ of an asymptotically flat and static manifold such that the Gauss curvature $K_\Sigma$ and the constant mean curvature $H_0$ of $\Sigma$ satisfy the inequality $4K_\Sigma\geq H_{0}^2$. In fact, they proved that $$\mathfrak{M}_{H}(\Sigma)\geq \frac{1}{8}\sqrt{\frac{H_{0}^2|\Sigma|}{4\pi}}\mathfrak{M}.$$
	Therefore, the Hawking mass is non-negative, according to the positive mass theorem.

	From a quasi-local mass point of view, it is desirable to draw information on the quasi-local mass of a surface $\Sigma$ purely from knowledge of the geometric data $(g_{\Sigma},\,H)$, where $g_{\Sigma}$ is the intrinsic metric on $\Sigma$ and $H$ is the mean curvature. It is natural to ask if the Hawking mass of such a surface is positive when the intrinsic metric is not far from being round. The positivity of the standard Hawking mass for CMC surfaces on compact $3$-manifolds  with non-negative scalar curvature was studied by \cite{miao2020} without assuming stability. However, additional conditions were imposed on geometric data.

	The Hawking mass can also be helpful in providing Item (1) for other definitions of quasi-local mass. A consequence of the proof of the Riemannian Penrose inequality via inverse mean
	curvature flow by \cite{huisken2001} gives us 
	\begin{eqnarray}\label{intog}
		\mathfrak{M}(M,\,g)\geq\mathfrak{M}_{H}(\partial\Omega),
	\end{eqnarray}
	where $(M^3,\,g)$ is an asymptotically flat Riemannian manifold (possibly with horizon boundary) with non-negative scalar curvature and $\Omega\subset M$ is a bounded open set with smooth topological boundary $\partial\Omega.$  
	
	The Hawking mass is often used as a lower bound for the Bartnik mass (a more delicate concept of quasi-local mass). Computing the Bartnik mass of a subset $\Omega$ is not easy, so one way to study the Bartnik mass is to look at its upper and lower bounds (cf. \cite{miao2020,mondino2022,lin2016}). We define the Bartnik mass 
	$\mathfrak{M}_{B}(\Omega)$ of $\Omega$ as 
	\begin{eqnarray}\label{bartnik}
		\mathfrak{M}_{B}(\Omega) = \inf\{\mathfrak{M}(\widetilde{M},\,\widetilde{g}):\, (\widetilde{M},\,\widetilde{g})\in\mathcal{A} \},  
	\end{eqnarray}
	where $\mathcal{A}$ is the set of asymptotically flat manifolds (possibly with horizon boundary) with non-negative scalar
	curvature into which $\Omega$ isometrically embeds such that $\partial\Omega\subset\widetilde{M}$ is outer-minimizing. The boundary $\partial\Omega$ is said to be outer-minimizing if $Perim(\Omega) \leq Perim(\Omega')$ for any set $\Omega'\subset M$ of a finite perimeter (denoted by $Perim(\Omega')$) and finite volume such that $\Omega\subset\Omega'$. The Bartnik mass is an important quasi-local mass since it satisfies the monotonicity property - Item (7). Note that the positive mass theorem (or the Riemannian Penrose inequality, in case all elements in $\mathcal{A}$ have non-empty horizon boundary) immediately yields the non-negativity of the
	Bartnik mass (see more in \cite{mondino2022} and in the references therein). In fact, since every smooth extension $(\widetilde{M},\,\widetilde{g})\in\mathcal{A}$ induces the same mean curvature on $\partial\Omega\subset\widetilde{M}$, the
	inequality \eqref{intog} combined with \eqref{bartnik} implies
	
	\begin{eqnarray*}
		\mathfrak{M}_{B}(\Omega)\geq\mathfrak{M}_{H}(\partial\Omega).
	\end{eqnarray*}

 	In \cite[Theorem 1.6]{mondino2022}, the authors assumed that the Hawking mass satisfies a certain local non-positivity condition to obtain a global rigidity theorem for Riemannian manifolds with non-negative scalar curvature. In fact, if the manifold is in addition asymptotically locally simply connected, the space must be isometric to $\mathbb{R}^3$. Moreover, to obtain a lower bound on the Bartnik mass \cite[Theorem 1.7]{mondino2022} out of the expansion of the Hawking mass obtained in \cite[Equation 11]{mondino2022}, the authors proved that the optimally perturbed geodesics spheres are outer-minimizing.

  We can see how useful it can be to prove lower bounds for the Hawking mass since computing the Bartnik mass of a subset is, in general, non-trivial. On the other hand,  from the definition, we can see that it is conceivable to expect upper bounds for the Bartnik quasi-local mass by direct comparison with somewhat explicit competitors. However, the issue of finding explicit lower bounds is more delicate. 
  
  Inspired by the above discussion, we will focus on finding sharp lower bounds for the charged Hawking mass of surfaces possessing an index equal to zero or one in non-compact (or compact) manifolds, particularly stable minimal surfaces, stable CMC surfaces, and minimal surfaces of index one. Moreover, we will provide examples proving that such lower bounds are sharp.

	\begin{definition}\cite{baltazar2023}\label{hawkingmass}
		Let $(M^3,\,g)$ be a three-dimensional Riemannian manifold and $\Sigma^2$ a closed surface on $M^3$. The charged Hawking mass is defined by
		\begin{eqnarray*}
			\mathfrak{M}_{CH}(\Sigma) = \sqrt{\frac{|\Sigma|}{16 \pi}}\left(\frac{1}{2}\mathfrak{X}(\Sigma) - \frac{1}{16\pi}\int_{\Sigma}(H^2 + \frac{4}{3}\Lambda)d\sigma + \frac{4\pi}{|\Sigma|}Q(\Sigma)^2 \right),
		\end{eqnarray*}
		where $|\Sigma|$ stands for the area of $\Sigma$. Here, $\mathfrak{X}(\Sigma)=2(1-g(\Sigma))$ and $g(\Sigma)$ stands for the genus of $\Sigma.$
	\end{definition}

	Note that, even for minimal surfaces, the positivity for the charged Hawking mass is not trivially obtained. It is a straightforward computation to show that the charged Hawking mass of a sphere in the Reissner-Nordstrom deSitter space satisfies $\mathfrak{M}_{CH}(\mathbb{S}(r))=\mathfrak{M}$ (cf. \cite{baltazar2023}). The result we will provide does not impose $(M^3,\,g)$ to be compact, and we do not ask for any additional condition over the mean curvature or Gauss curvature of a surface $\Sigma$ in the electrostatic system. It is well-known that if a surface $\Sigma$ have $Index(\Sigma)=0$, then for any $\phi\in C^{\infty}(\Sigma)$ we have
	\begin{eqnarray}\label{stabcond}
		0\leq \int_{\Sigma}|\nabla_{\Sigma}\phi|^2d\sigma - \int_{\Sigma}( Ric(\nu,\,\nu) + |A|^2)\phi^2d\sigma.
	\end{eqnarray}
	Taking $\phi=1$ we get
	\begin{eqnarray}\label{1}
		\int_{\Sigma}(Ric(\nu,\,\nu)+|A|^2)d\sigma\leq0.
	\end{eqnarray}

	\begin{theorem}\label{Teo1}
		Let $(M^3,\,g,\,f,\,E)$ be a three-dimensional electrostatic system (compact or non-compact) and $\Sigma$ a closed surface on $M$ with $Index(\Sigma)=0$. Then, the genus of $\Sigma$ is at most $1$. Moreover, 
		\begin{eqnarray}\label{ine1}
			\mathfrak{M}_{CH}(\Sigma)  \geq \left(\frac{|\Sigma|}{4\pi}\right)^{1/2}\left[\frac{1}{16\pi}\int_{\Sigma}H^2d\sigma +\frac{4\pi}{|\Sigma|}Q(\Sigma)^2+\frac{\Lambda}{3}\left(\frac{|\Sigma|}{4\pi}\right)\right].
		\end{eqnarray}
		The equality holds if and only if $\Sigma$ is totally umbilical, $E$ is orthogonal to $\Sigma$, and equality holds in \eqref{1}. 
	\end{theorem}

	\begin{remark}
		In Proposition \ref{CMCstable} (below), we will prove that any CMC surface in the RNdS space must be a sphere. The equality in Theorem \ref{Teo1} holds for a sphere of radius \eqref{raiophoton} in the Reissner-Nordstrom deSitter space. Moreover, we can conclude that the charged Hawking mass is non-negative if $\Lambda\geq0$.
	\end{remark}
	
	We will apply the above theorem to a stable minimal surface to get rigidity (Theorem \ref{sharptheo}). Nonetheless, we will also provide the lower bound for the charged Hawking mass considering a stable CMC surface with non-zero mean curvature and a minimal surface of index one (cf. Theorem \ref{cocoro1} and Theorem \ref{cocoro2}).

	Some important results use the Hawking mass to obtain the rigidity of a given space \cite{baltazar2023,maximo2013}. In \cite{baltazar2023}, the authors proved the rigidity of $3$-manifolds satisfying Einstein's constraints, assuming the existence of a strictly stable minimal two-sphere that locally maximizes the charged Hawking mass, concluding that the ambient space must be locally a copy of RNdS, see also \cite{maximo2013} for the analogous result without charge. If we consider just stability, the rigidity for area-minimizing two-sphere holds if the two-sphere has a fixed area. In this case, the sphere is a global maximum of the Hawking mass \cite[Remark 4.4]{maximo2013}. We will assume that \cite[Theorem 2]{baltazar2023} holds for area-minimizing two-sphere with a fixed area (see the details in the proof of Theorem \ref{sharptheo}).

 In the following theorem, we will use the results of \cite{baltazar2023,maximo2013} to obtain a sharp lower bound for the charged Hawking mass of minimal surfaces.


	\begin{theorem}\label{sharptheo}
		Let $(M^3,\,g,\,f,\,E)$ be a three-dimensional electrostatic system with a non-null cosmological constant and $\Sigma\subset M$ an embedded
		stable minimal surface. Then, the genus of $\Sigma$ is at most $1$. Moreover, 
		\begin{eqnarray}\label{desiminimal}
			\mathfrak{M}_{CH}(\Sigma)  \geq \left(\frac{|\Sigma|}{4\pi}\right)^{1/2}\left[\frac{4\pi}{|\Sigma|}Q(\Sigma)^2+\frac{\Lambda}{3}\left(\frac{|\Sigma|}{4\pi}\right)\right].
		\end{eqnarray}
		In addition, if $\Sigma$ is a stable minimal two-sphere, equality holds if and only if $\Sigma$ is the horizon boundary of the ultracold black hole system.
	\end{theorem}

	

	Ahead, we want to obtain a lower bound for the charged Hawking mass for a stable CMC surface ($H\neq0$), i.e., a constant mean curvature surface $\Sigma$ satisfying the stability condition \eqref{stabcond} for any function $\phi\in C^{\infty}(\Sigma)$ such that $\int_{\Sigma}\phi d\sigma=0$. However, we can not provide the rigidity since, as far as we know, we do not possess a result like the main theorem in \cite{baltazar2023} for stable CMC surfaces. It is well-known that if a surface $\Sigma$ is a stable CMC surface, then
	\begin{eqnarray}\label{22}
		\int_{\Sigma}(Ric(\nu,\,\nu)+|A|^2)d\sigma\leq8\pi,
	\end{eqnarray}
	see the details in \cite{christodoulou1986,li82}.
	
	\begin{theorem}\label{cocoro1}
		Let $(M^3,\,g,\,f,\,E)$ be a three-dimensional electrostatic system with $\Sigma$ a closed stable CMC surface in $M$. Then, the genus of $\Sigma$ is at most $3$. Moreover,
		\begin{eqnarray}\label{ine3}
			\mathfrak{M}_{CH}(\Sigma)  \geq \left(\frac{|\Sigma|}{4\pi}\right)^{1/2}\left[\frac{H^2}{16\pi}|\Sigma|+\frac{4\pi}{|\Sigma|}Q(\Sigma)^2+\frac{\Lambda}{3}\left(\frac{|\Sigma|}{4\pi}\right) - 1\right].
		\end{eqnarray}
		The equality holds if and only if $\Sigma$ is totally umbilical, $E$ is orthogonal to $\Sigma$, and equality holds in \eqref{22}.
	\end{theorem}

	\begin{remark}
		Considering $\Lambda>0$, the charged Hawking mass is non-negative for any stable CMC surface in the RNdS space such that $|\Sigma|\geq \dfrac{12\pi}{\Lambda}$. The Nariai system is $M^3=[0,\,\pi/\sqrt{\Lambda}]\times\mathbb{S}^2$ with metric tensor $g=ds^2 + \dfrac{1}{\Lambda}g_{\mathbb{S}^2}$, where $f(s)=\sin(\sqrt{\Lambda}s)$, $\mathfrak{M}=\dfrac{1}{3\sqrt{\Lambda}}$ and $Q=0$. Therefore, the area of a two-sphere $\Sigma$ in the Nariai system is $|\Sigma| = \dfrac{4\pi}{\Lambda}$. The equality holds in Theorem \ref{cocoro1} for stable CMC sphere such that $s=\dfrac{\pi}{2\sqrt{\Lambda}}.$ This will be discussed in more detail in the proof of Theorem \ref{cocoro1}.

	\end{remark}

	The last theorem above (Theorem \ref{cocoro1}) inspired us to pursue a similar result for minimal surfaces of index one. It is well-known that such a surface must satisfy the following inequality:
	\begin{eqnarray}\label{223}
		\int_{\Sigma}(Ric(\nu,\,\nu)+|A|^2)d\sigma\leq8\pi\left(1+ Int\left[\frac{1+g(\Sigma)}{2}\right]\right).
	\end{eqnarray}
	Here, $Int[x]$ denotes the integer part of $x$. See details in \cite{tiarlos,ritore92}.
	
	\begin{theorem}\label{cocoro2}
		Let $(M^3,\,g,\,f,\,E)$ be a three-dimensional electrostatic system with $\Sigma$ a closed minimal surface of index one in $M$. Then,
		\begin{eqnarray}\label{ine4}
			\mathfrak{M}_{CH}(\Sigma)  \geq \left(\frac{|\Sigma|}{4\pi}\right)^{1/2}\left(\frac{4\pi}{|\Sigma|}Q(\Sigma)^2+\frac{\Lambda}{3}\left(\frac{|\Sigma|}{4\pi}\right) - 1-Int\left[\frac{1+g(\Sigma)}{2}\right]\right).
		\end{eqnarray}
		The equality holds if and only if $\Sigma$ is totally umbilical, $E$ is orthogonal to $\Sigma$, and equality holds in \eqref{223}.
	\end{theorem}
	
	\begin{remark}
		
		Considering $\Lambda>0$, the charged Hawking mass is non-negative for any minimal surface of index one in the RNdS space such that $|\Sigma|\geq \dfrac{12\pi}{\Lambda}\left( 1+Int\left[\frac{1+g(\Sigma)}{2}\right]\right)$.
		Consider the deSitter system which is $M^3=\mathbb{S}^3_{+}$ with metric $g=\dfrac{3}{\Lambda}g_{\mathbb{S}^3}$, $\mathfrak{M}=Q=0$ and $f(x)=x_4$, with Cartesian coordinates $x=(x_1,\,x_2,\,x_3,\,x_4)$. We can see that the equator is a minimal sphere $\Sigma$ of index one with area $|\Sigma|=\dfrac{12\pi}{\Lambda}.$ Therefore, the equality holds in Theorem \ref{cocoro2}.

	\end{remark}

	\section{Proof of the main results}

	In this section, we present the proof of the main results. Before we prove our main results concerning a lower bound for the charged Hawking mass, we will prove a criterion for CMC surfaces in the RNdS space to be stable (Proposition \ref{CMCstable}). This result will be important for a better understanding of Theorem \ref{Teo1}. Then, we will prove the main result of the paper, which is the core of the proof for the sharp lower bounds for the charged Hawking mass of stable surfaces in the electrostatic space-time (Proposition \ref{prop2}).
	
	\begin{proposition}\label{CMCstable}
		Any CMC surface in the three-dimensional Reissner-Nordstrom deSitter space must be a sphere $\mathbb{S}^2(r)$ with mean curvature $H$ given by 
		$$H=\frac{2}{r}\left(1-\frac{2\mathfrak{M}}{r}+\frac{Q^2}{r^2}-\frac{\Lambda r^2}{3}\right)^{1/2}.$$ 
		Moreover, if
		\begin{eqnarray*}
			\frac{3\mathfrak{M}}{2}\left(1 - \sqrt{1-\frac{8Q^2}{9\mathfrak{M}^2}}\right)\leq r\leq  \frac{3\mathfrak{M}}{2}\left(1 + \sqrt{1-\frac{8Q^2}{9\mathfrak{M}^2}}\right)
		\end{eqnarray*}
		such CMC surface must be stable.
	\end{proposition}
	\begin{proof}[{\bf Proof of Proposition \ref{CMCstable}.}]
		We start with a result due to Brendle proving that any CMC surface must be $\mathbb{S}^2(r)$ in the Reissner-Nordstrom deSitter space (cf. \cite[Section 5]{brendle2012} to see that $\Lambda$ plays no role in the proof of Corollary 1.3).  For every spherical slice, $\Sigma=\{r\}\times\mathbb{S}^2$ in the RNdS space, we have $H(r)=\dfrac{2}{r}f(r)$ and $|\Sigma|=4\pi r^2$. By the spherical symmetry, $\Sigma$ is umbilical. Moreover, $E=\dfrac{Q}{r^2}f(r)\partial r$ is parallel to normal vector field $\nu=f(r)\partial_r.$ 
		
		Remember that
		\begin{eqnarray*}
			\Delta f = \Delta^\Sigma f + \nabla^2 f(\nu,\,\nu) + H\langle \nabla f,\,\nu\rangle.
		\end{eqnarray*}
		By Definition \ref{def1} we have 
		\begin{eqnarray}\label{vvv}
			H\langle \nabla f,\,\nu\rangle - f|A|^2 &=& \Delta f -\Delta^\Sigma f - \nabla^2 f(\nu,\,\nu) - f|A|^2 \nonumber\\
			&=& (|E|^2-\Lambda)f - \Delta^\Sigma f - f[Ric(\nu,\,\nu) + 2\langle E,\,\nu\rangle^2 - (|E|^2 + \Lambda)]- f|A|^2\nonumber\\
			&=& 2f(|E|^2 - \langle E,\,\nu\rangle^2) - \Delta^\Sigma f - fRic(\nu,\,\nu) - f|A|^2 .
		\end{eqnarray}
		Since $E$ is parallel to $\nu$ (unitary), i.e., $|E|^2 =\langle E,\,\nu\rangle^2$, we obtain 
		\begin{eqnarray*}
			\left(\frac{H}{f}\langle \nabla f,\,\nu\rangle - |A|^2\right)f
			= (- \Delta^\Sigma  - Ric(\nu,\,\nu) - |A|^2)f .
		\end{eqnarray*}
		i.e.,
		\begin{eqnarray}\label{danadana}
			\frac{H}{f}\langle \nabla f,\,\nu\rangle - |A|^2 
			=  - Ric(\nu,\,\nu) - |A|^2,
		\end{eqnarray}
		where we use that $\Delta^{\Sigma}f=0$, since $f$ is a non-null constant at $\Sigma=\{r\}\times\mathbb{S}^2$.

		In the Reissner-Nordstrom deSitter space, we have
		\begin{eqnarray*}
			|E|^2 = \frac{Q^2}{r^4}
		\end{eqnarray*}
		and 
		\begin{eqnarray*}
			\nabla f = f(r)^2f'(r)\partial_r= f(r)^2\frac{1}{2f(r)}\left(\frac{2\mathfrak{M}}{r^2}-\frac{2Q^2}{r^3}-\frac{2\Lambda r}{3}\right)\partial_r.
		\end{eqnarray*}
		Thus, 
		\begin{eqnarray*}
			\frac{H}{f}\langle \nabla f,\,\nu\rangle = \frac{2}{r}\left(\frac{\mathfrak{M}}{r^2}-\frac{Q^2}{r^3}-\frac{\Lambda r}{3}\right).
		\end{eqnarray*}

		Moreover, 
		\begin{eqnarray}\label{cmcstable}
			\frac{H}{f}\langle \nabla f,\,\nu\rangle -|A|^2&=& \frac{2}{r}\left(\frac{\mathfrak{M}}{r^2}-\frac{Q^2}{r^3}-\frac{\Lambda r}{3}\right) - \frac{H^2}{2}= \frac{2}{r}\left(\frac{\mathfrak{M}}{r^2}-\frac{Q^2}{r^3}-\frac{\Lambda r}{3}\right) - \frac{1}{2}\frac{4}{r^2}f(r)^2\nonumber\\
			&=&\frac{2}{r}\left[\frac{\mathfrak{M}}{r^2}-\frac{Q^2}{r^3}-\frac{\Lambda r}{3} - \frac{1}{r}\left(1-\frac{2\mathfrak{M}}{r}+\frac{Q^2}{r^2}-\frac{\Lambda r^2}{3}\right)\right]\nonumber\\
			&=& \frac{2}{r}\left[\frac{3\mathfrak{M}}{r^2}-\frac{2Q^2}{r^3} - \frac{1}{r}\right] = - \frac{2}{r^4}\left[2Q^2  -3\mathfrak{M}r + r^2\right].
		\end{eqnarray}
		Note that if $\mathfrak{M}=Q=0$ we always have an unstable sphere (see the deSitter space). 
		We can conclude that if 
		\begin{eqnarray}\label{raiostable}
			\frac{3\mathfrak{M}}{2}\left(1 - \sqrt{1-\frac{8Q^2}{9\mathfrak{M}^2}}\right)\leq r\leq  \frac{3\mathfrak{M}}{2}\left(1 + \sqrt{1-\frac{8Q^2}{9\mathfrak{M}^2}}\right)
		\end{eqnarray}
		the CMC surface must be stable. 
		Moreover, we can see that the photon sphere of radius $\frac{3\mathfrak{M}}{2}\left(1 + \sqrt{1-\frac{8Q^2}{9\mathfrak{M}^2}}\right)$ must be a stable CMC surface in the RNdS space. In fact, from \eqref{raiostable} for any $\phi\in C^{\infty}(\Sigma)$ such that $\int_{\Sigma}\phi d\sigma = 0$ we have
		\begin{eqnarray*}
			\int_{\Sigma}|\nabla_{\Sigma}\phi|^2 d\sigma - \int_{\Sigma}( Ric(\nu,\,\nu) + |A|^2)\phi^2 d\sigma
			&=& \int_{\Sigma}|\nabla_{\Sigma}\phi|^2 d\sigma \\
			&+& \int_{\Sigma}\left(\frac{H}{f}\langle \nabla f,\,\nu\rangle - |A|^2 \right)\phi^2 d\sigma\geq0.
		\end{eqnarray*}

	\end{proof}
	
	\begin{proposition}\label{prop2}
		Let $(M^3,\,g,\,f,\,E)$ be a three-dimensional electrostatic system (compact or non-compact) and $\Sigma$ a closed surface on $M$ satisfying $$\int_{\Sigma}(Ric(\nu,\,\nu)+|A|^2)d\sigma\leq \mathcal{C}.$$ Then,
		\begin{eqnarray*}
			\mathfrak{M}_{CH}(\Sigma)  \geq \left(\frac{|\Sigma|}{4\pi}\right)^{1/2}\left[\frac{1}{16\pi}\int_{\Sigma}H^2d\sigma +\frac{4\pi}{|\Sigma|}Q(\Sigma)^2+\frac{\Lambda}{3}\left(\frac{|\Sigma|}{4\pi}\right) - \frac{\mathcal{C}}{8\pi}\right],
		\end{eqnarray*}
		where $\mathcal{C}\in\mathbb{R}.$ Moreover, the genus of $\Sigma$ is bounded from above, i.e.,
$$1+\dfrac{\mathcal{C}}{4\pi}\geq g(\Sigma).$$
	\end{proposition}
	\begin{proof}[{\bf Proof of Proposition \ref{prop2}.}]
		From Definition \ref{def1} and the Gauss equation we have, respectively,
		\begin{eqnarray*}
			\nabla^2 f(\nu,\,\nu) = f[Ric(\nu,\,\nu) + 2\langle E,\,\nu\rangle^2 - (|E|^2 + \Lambda)]
		\end{eqnarray*}
		and 
		\begin{eqnarray*}
			\frac{R}{2} = K + Ric(\nu,\,\nu) + \frac{1}{2}(|A|^2 - H^2),
		\end{eqnarray*}
		where $K$ is Gauss curvature of $\Sigma.$
		Combining the above equations and using \eqref{rrr}, we obtain 
		\begin{eqnarray*}
			\nabla^2 f(\nu,\,\nu) = f[ \frac{1}{2}(H^2 - |A|^2) - K + 2\langle E,\,\nu\rangle^2].
		\end{eqnarray*}

		Now, from 
		\begin{eqnarray}\label{decopLapla}
			\Delta f = \Delta^\Sigma f + \nabla^2 f(\nu,\,\nu) + H\langle \nabla f,\,\nu\rangle
		\end{eqnarray}
		and the last equation we deduce 
		\begin{eqnarray*}
			\frac{\Delta f}{f} = \frac{\Delta^\Sigma f}{f} + [ \frac{1}{2}(H^2 - |A|^2) - K + 2\langle E,\,\nu\rangle^2] + \frac{H}{f}\langle \nabla f,\,\nu\rangle.
		\end{eqnarray*}
		By integrating the above identity, we have
		\begin{eqnarray}\label{show}
			\int_{\Sigma}(|E|^2 - \Lambda)d\sigma &=&\int_{\Sigma}\frac{H}{f}\langle \nabla f,\,\nu\rangle d\sigma + \int_{\Sigma}\frac{|\nabla_{\Sigma}f|^2}{f^2}d\sigma \nonumber\\
   &+& \int_{\Sigma}\left[ \frac{1}{2}(H^2 - |A|^2) - K + 2\langle E,\,\nu\rangle^2\right]d\sigma.
		\end{eqnarray}

		We consider that
		\begin{eqnarray}\label{stata}
			\int_{\Sigma}( Ric(\nu,\,\nu) + |A|^2)d\sigma \leq \mathcal{C},
		\end{eqnarray}
  where $\mathcal{C}\in\mathbb{R}.$
		Dividing \eqref{vvv} by $f$ and by using \eqref{stata} we get
		\begin{eqnarray}\label{como1}
			\int_{\Sigma}\left(\frac{H}{f}\langle \nabla f,\,\nu\rangle +\frac{1}{f^2}|\nabla_\Sigma f|^2\right)d\sigma
			&\geq&   \int_{\Sigma}|A|^2d\sigma \nonumber\\
			&+& 2\int_{\Sigma}(|E|^2 - \langle E,\,\nu\rangle^2)d\sigma - \mathcal{C}.
		\end{eqnarray}

		Thus, from \eqref{show}, \eqref{como1} and the Gauss-Bonnet theorem we have
		\begin{eqnarray*}
			4\pi(1-g(\Sigma))&\geq& \int_{\Sigma}[ \frac{1}{2}(H^2 + |A|^2)]d\sigma + \int_{\Sigma}(|E|^2 + \Lambda)d\sigma - \mathcal{C},
		\end{eqnarray*} 
		and since $2|A|^2\geq H^2$, 
		\begin{eqnarray}\label{ggggenus}
			4\pi(1-g(\Sigma)) &\geq& \frac{3}{4}\int_{\Sigma}H^2d\sigma + \int_{\Sigma}(|E|^2 + \Lambda)d\sigma - \mathcal{C}.
		\end{eqnarray}
		With this inequality, we can analyze the topology of $\Sigma$. Indeed, $4\pi(1-g(\Sigma))+\mathcal{C}\geq0$, i.e., $$1+\dfrac{\mathcal{C}}{4\pi}\geq g(\Sigma).$$

		Combine the above inequality \eqref{ggggenus} with Definition \ref{hawkingmass}, i.e.,
		\begin{eqnarray*}
			2\pi\mathfrak{X}(\Sigma)  = 4\pi\sqrt{\frac{16\pi}{|\Sigma|}}\mathfrak{M}_{CH}(\Sigma)   + \frac{1}{4}\int_{\Sigma}(H^2 + \frac{4}{3}\Lambda)d\sigma - \frac{16\pi^2}{|\Sigma|}Q(\Sigma)^2  ,
		\end{eqnarray*}
		to obtain
		\begin{eqnarray*}
			4\pi\sqrt{\frac{16\pi}{|\Sigma|}}\mathfrak{M}_{CH}(\Sigma)     \geq  \frac{16\pi^2}{|\Sigma|}Q(\Sigma)^2 + \frac{1}{2}\int_{\Sigma}H^2d\sigma + \int_{\Sigma}|E|^2 d\sigma + \frac{2}{3}\Lambda|\Sigma| - \mathcal{C}.
		\end{eqnarray*}
		Moreover, applying H\"older's inequality to the definition of charge given by \eqref{chaaaarge} yields to
		\begin{eqnarray*}
		16\pi^2Q(\Sigma)^2=\left(\int_{\Sigma}\langle E,\,\nu\rangle d\sigma\right )^2  \leq|\Sigma|\int_{\Sigma}\langle E,\,\nu\rangle^2d\sigma\leq|\Sigma|\int_{\Sigma}|E|^2d\sigma.
		\end{eqnarray*}

		Finally, 
		\begin{eqnarray*}
			4\pi\sqrt{\frac{16\pi}{|\Sigma|}}\mathfrak{M}_{CH}(\Sigma)     \geq  \frac{32\pi^2}{|\Sigma|}Q(\Sigma)^2 + \frac{1}{2}\int_{\Sigma}H^2d\sigma + \frac{2}{3}\Lambda|\Sigma| - \mathcal{C}.
		\end{eqnarray*}
		The inequality above proves the theorem.

	\end{proof}

	\begin{proof}[{\bf Proof of Theorem \ref{Teo1}.}]
		By hypothesis, we have
		\begin{eqnarray}\label{indexxxxxx0}
			0\geq   \int_{\Sigma}( Ric(\nu,\,\nu) + |A|^2)d\sigma.
		\end{eqnarray}
		Therefore, Proposition \ref{prop2} holds for $\mathcal{C}=0$, and from \eqref{ggggenus} we have $g(\Sigma)\leq1.$ 
		
		Let us apply \eqref{ine1} to a sphere in the Reissner-Nordstrom deSitter space. We know that in this case $\mathfrak{M}_{CH}(\mathbb{S}(r)) =\mathfrak{M}$, see \cite{baltazar2023}. Since
		\begin{eqnarray*}
			\frac{1}{16\pi}\int_{\mathbb{S}(r)}H^2d\sigma = f(r)^2
		\end{eqnarray*}
		our inequality \eqref{ine1} is reduced to
		\begin{eqnarray*}
			\mathfrak{M}\geq r -  2\mathfrak{M} +2\frac{Q^2}{r},
		\end{eqnarray*}
		where the equality holds for $$r= \frac{3\mathfrak{M}}{2}\left(1 \pm \sqrt{1-\frac{8Q^2}{9\mathfrak{M}^2}}\right).$$
		This is a totally umbilical stable constant mean curvature surface in the RNdS space in which equality holds in \eqref{indexxxxxx0}, and $E$ is parallel to $\nu$. See Proposition \ref{CMCstable}, equations \eqref{danadana} and \eqref{cmcstable}.

	\end{proof}

	\begin{proof}[{\bf Proof of Theorem \ref{sharptheo}.}]
		
		Consider a stable minimal sphere of radius $r$ in the RNdS space and applying the inequality presented by Theorem \ref{Teo1}, we get        
		\begin{eqnarray*}
			\mathfrak{M}  \geq \frac{Q^2}{r} + \frac{\Lambda}{3}r^3,
		\end{eqnarray*}
		where equality holds if and only if
		\begin{eqnarray}\label{poly}
			\frac{\Lambda}{3}r^4 -\mathfrak{M}r + Q^2 = 0.
		\end{eqnarray}
		The mean curvature of a two-sphere in the RNdS space is given by $H(r)=\dfrac{2}{r}f(r)$. Moreover, the ADM mass is 
		\begin{eqnarray}\label{masssss}
			\mathfrak{M} = \frac{1}{2}\left(r + \frac{Q^2}{r} - \frac{\Lambda}{3}r^3 \right).
		\end{eqnarray}
		Hence, combining \eqref{poly} and \eqref{masssss} we get
		\begin{eqnarray*}\label{aiai}
			\Lambda r^4 - r^2 + Q^2 = 0, 
		\end{eqnarray*}
		i.e.,
		\begin{eqnarray*}
			r^2 = \frac{1\pm\sqrt{1-4\Lambda Q^2}}{2\Lambda}.
		\end{eqnarray*}

		We can consult \cite[Section 3.8]{tiarlos} to conclude that a sphere in the RNdS space is stable when the radius $r$ is bounded, i.e.,
		\begin{eqnarray}\label{tobestable}
			\frac{1-\sqrt{1-4\Lambda Q^2}}{2\Lambda} \leq r^2 \leq \frac{1+\sqrt{1-4\Lambda Q^2}}{2\Lambda},
		\end{eqnarray}
		and it is strictly stable if
		\begin{eqnarray*}
			\frac{1-\sqrt{1-4\Lambda Q^2}}{2\Lambda} < r^2 < \frac{1+\sqrt{1-4\Lambda Q^2}}{2\Lambda}.
		\end{eqnarray*}
		In fact, the Jacobi operator for the RNdS space is given by 
		$$J= - \Delta^{\Sigma} - \frac{1}{r^4}(\Lambda r^4 -r^2 +Q^2).$$

		On the other hand, consider the equality holds in  \eqref{desiminimal} for a stable minimal sphere of radius $r$ in the electrostatic system. Thus,
		\begin{eqnarray*}
			\mathfrak{M}_{CH}(\Sigma)  = \left(\frac{|\Sigma|}{4\pi}\right)^{1/2}\left[\frac{4\pi}{|\Sigma|}Q(\Sigma)^2+\frac{\Lambda}{3}\left(\frac{|\Sigma|}{4\pi}\right)\right] = \frac{Q^2}{r} + \frac{\Lambda}{3}r^3.
		\end{eqnarray*}
		We must have a priori 
		\begin{eqnarray*}
			r^2 = \frac{1\pm\sqrt{1-4\Lambda Q^2}}{2\Lambda},\quad\mbox{i.e.,}\quad Q^2 = \frac{1}{4\Lambda}[1 - (1 - 2\Lambda r^2)^2]
		\end{eqnarray*}
		Otherwise, the equation \eqref{poly} does not hold.

		Hence, 
		\begin{eqnarray*}
			\mathfrak{M}_{CH}(\Sigma) = \frac{Q^2}{r} + \frac{\Lambda}{3}r^3 =  \frac{1}{4\Lambda r}[1 - (1 - 2\Lambda r^2)^2] + \frac{\Lambda}{3}r^3 = r - \frac{2}{3}\Lambda r^3.
		\end{eqnarray*}
		Therefore, the sphere $\Sigma=\mathbb{S}(r)$ of radius $r^2=\dfrac{1}{2\Lambda}$ maximizes the charged Hawking mass, which is strictly stable. Hence, from Theorem 2 in \cite{baltazar2023}, there is a neighborhood of such sphere in $(M,\,g)$ that is isometric to the RNdS space $((-\varepsilon,\,\varepsilon)\times\Sigma,\,g_{RNdS})$ for some $\varepsilon>0$. Since in the RNdS, $\Sigma$ is minimal if and only if $f=0$, such surface must be the horizon boundary of $M$, i.e., $\Sigma  = \partial M$. 
		
		Considering that the roots of $f$ in the RNdS space are $r_+=r_-=r_c,$ we can perform a change of variable such that $g_{RNdS} = ds^2 + \rho^2 g_{\mathbb{S}^2}$ with $f(s)=s$, $\rho^2=\dfrac{1}{2\Lambda}$, $Q^2=\dfrac{1}{4\Lambda}$ and $\mathfrak{M}^2 = \dfrac{2}{9\Lambda}$, which correspond to the ultracold black hole system, where $M=[0,\,+\infty)\times\mathbb{S}^2$ (cf. \cite[Equation 1.12]{baltazar2023} and \cite[Section 3.7]{tiarlos}), with a stable horizon at $s=0$. Here, we are assuming the change of variable given by \eqref{novacoord}.

		From \cite[Remark 4.4]{maximo2013}, the rigidity for area-minimizing two-sphere holds if the area of $\Sigma$ is fixed. In this case, $\Sigma$ is a global maximum of the charged Hawking mass. We will assume that Theorem 2 in \cite{baltazar2023} holds for a stable minimal sphere. In fact, for $\Lambda>0$, to bypass \cite[Proposition 4]{baltazar2023} we need to impose that
		\begin{eqnarray*}(\underbrace{\lambda_1(J)}_{=0}+\Lambda)|\Sigma| + \frac{16\pi^2Q(\Sigma)^2}{|\Sigma|}= 4\pi,
		\end{eqnarray*}
		where $\lambda_1(J)$ stands for the first eigenvalue of the Jacobi operator. To get this equality in our settings in the ultracold black hole system we must take $|\Sigma|=4\pi\rho^2 = \frac{2\pi}{\Lambda}$ and $Q^2=\frac{1}{4\Lambda}$. 
  	\end{proof}

	\begin{proof}[{\bf Proof of Theorem \ref{cocoro1}.}]
		A stable CMC surface satisfies 
		\begin{eqnarray*}
			8\pi\geq \int_{\Sigma}\left( |A|^2 + Ric(\nu, \nu)\right)\, d\sigma.
		\end{eqnarray*}
		See more details about the above inequality in \cite{christodoulou1986}. So, applying Proposition \ref{prop2} and considering $\mathcal{C}=8\pi$ we obtain inequality \eqref{ine3}.
		Assuming $\Lambda\geq0$ from \eqref{ggggenus} we get $g(\Sigma)\leq3$.

		Let us apply \eqref{ine3} to a sphere in the RNdS. We know that in this case $\mathfrak{M}_{CH}(\mathbb{S}(r)) =\mathfrak{M}$, see \cite{baltazar2023}. Moreover, from \eqref{cmcstable} if $Q=0$, $\mathfrak{M}=\dfrac{1}{3\sqrt{\Lambda}}$ and $r=\dfrac{1}{\sqrt{\Lambda}}$, then such sphere must be stable. Furthermore, 
		\begin{eqnarray*}
			\frac{1}{16\pi}\int_{\mathbb{S}(r)}H^2d\sigma = f(r)^2.
		\end{eqnarray*}
		So, 
		\begin{eqnarray}\label{fim}
			\mathfrak{M} \geq \left(\frac{|\Sigma|}{4\pi}\right)^{1/2}\left[f^2 + \frac{4\pi}{|\Sigma|}Q^2  + \frac{\Lambda}{3}\frac{|\Sigma|}{4\pi}  - 1\right].
		\end{eqnarray}
		Considering the coordinate transformation \eqref{novacoord}, see \cite[Section 3.4]{tiarlos}, the Nariai system is $M^3=[0,\,\pi/\sqrt{\Lambda}]\times\mathbb{S}^2$ with metric tensor $g=ds^2 + \dfrac{1}{\Lambda}g_{\mathbb{S}^2}$, where $f(s)=\sin(\sqrt{\Lambda}s)$, $\mathfrak{M}=\dfrac{1}{3\sqrt{\Lambda}}$ and $Q=0$. The equality in \eqref{fim} holds for $s=\dfrac{\pi}{2\sqrt{\Lambda}}.$
  \end{proof}

	\begin{proof}[{\bf Proof of Theorem \ref{cocoro2}.}]
		It is well-known that a closed minimal surface of index one satisfies the inequality:
		\begin{eqnarray*}
			8\pi\left(1+ Int\left[\frac{1+g(\Sigma)}{2}\right]\right)\geq \int_{\Sigma}\left( |A|^2 + Ric(\nu, \nu)\right)\, d\sigma,
		\end{eqnarray*}
		where $Int[x]$ denotes the integer part of $x.$ See for instance \cite[Proposition 17]{tiarlos} and \cite{ritore92}. Considering $\mathcal{C}=8\pi\left(1+ Int\left[\dfrac{1+g(\Sigma)}{2}\right]\right)$ in Proposition \ref{prop2} will give us the desired result.
		
		Consider the deSitter system which is $M^3=\mathbb{S}^3_{+}$ with metric $g=\dfrac{3}{\Lambda}g_{\mathbb{S}^3}$, $\mathfrak{M}=Q=0$ and $f(x)=x_4$, where $x=(x_1,\,x_2,\,x_3,\,x_4)$. We can see that the equator is a minimal sphere $\Sigma$ of index one with area $|\Sigma|=\dfrac{12\pi}{\Lambda}.$ Therefore, the equality holds in \eqref{ine4}.
		\end{proof}

	
	\

	\noindent{\bf Conflict of interest:} The authors declare no conflict of interest.
	
	\

	\noindent{\bf Data Availability:} Not applicable.
	
	\
	
	\noindent{{\bf Acknowledgments.}} The authors wish to express their gratitude to Professor Ivaldo Nunes for his valuable comments and discussions. The authors thank the anonymous referee for thoroughly reviewing the manuscript and for the valuable suggestions.

\end{document}